\documentclass[a4paper,12pt]{article}
\usepackage[top=2.5cm,bottom=2.5cm,left=2.5cm,right=2.5cm]{geometry}
\usepackage{cite,amsmath}
    \usepackage[margin=1cm,%
                font=small,%
                format=hang,%
                labelsep=period,%
                labelfont=bf]{caption}

\usepackage{bm}
\usepackage[algoruled,linesnumbered]{algorithm2e}
\usepackage{algorithmic}

\usepackage{amsfonts,epsf,here,graphicx}
\usepackage{latexsym,multirow,rotating}
\usepackage{pgf,tikz,subfigure}
\usepackage{epstopdf}

\newtheorem{theorem}{\bf Theorem}[section]

\newtheorem{lemma}[theorem]{\bf Lemma}
\newtheorem{proposition}[theorem]{\bf Proposition}

\newtheorem{definition}[theorem]{\bf Definition}

\newcommand{\qed}{\hfill $\square$ \bigskip}

\begin{document}

\baselineskip=0.30in
\vspace*{40mm}

\begin{center}
{\LARGE \bf General cut method for computing Szeged-like topological indices with applications to molecular graphs}
\bigskip \bigskip

{\large \bf Simon Brezovnik$^{a,b}$, \qquad
Niko Tratnik$^{a}$
}
\bigskip\bigskip

\baselineskip=0.20in

$^a$\textit{Faculty of Natural Sciences and Mathematics, University of Maribor, Slovenia} \\
{\tt simon.brezovnik2@um.si, niko.tratnik@um.si}
\medskip

$^b$\textit{Faculty of Education, University of Maribor, Slovenia}

\bigskip\medskip

(Received \today)

\end{center}

\noindent
\begin{center} {\bf Abstract} \end{center}
Szeged, PI and Mostar indices are some of the most investigated distance-based molecular descriptors. Recently, many different variations of these topological indices appeared in the literature and sometimes they are all together called Szeged-like topological indices. In this paper, we formally introduce the concept of a general Szeged-like topological index, which includes all mentioned indices and also infinitely many other topological indices that can be defined in a similar way. As the main result of the paper, we provide a cut method for computing a general Szeged-like topological index for any strength-weighted graph. This greatly generalizes various methods known for some of the mentioned indices and therefore rounds off such investigations. 
Moreover, we provide applications of our main result to benzenoid systems, phenylenes, and coronoid systems, which are well-known families of molecular graphs. In particular, closed-form formulas for some subfamilies of these graphs are deduced.
\vspace{3mm}\noindent


\baselineskip=0.30in



\section{Introduction}
In mathematical chemistry and in chemoinformatics, many numerical quantities of molecular graphs have been introduced and studied in order to describe various properties of molecules. Such graph invariants are most commonly referred to as molecular descriptors. They are used for the development of quantitative structure-activity relationships (QSAR) and quantitative structure-property relationships (QSPR) in which some properties of compounds are correlated with their chemical structure. However, in recent years similar descriptors have found enormous applications also in a rapidly growing research of complex networks \cite{estrada}, for example communications networks, social networks, biological networks, etc. Whenever a graph invariant is used for describing molecular structure or network topology, we usually call it a topological index or topological descriptor. 

One of the oldest topological indices is the famous \textit{Wiener index}, which is defined for a connected graph $G$ as 
$$W(G) = \sum_{\lbrace u,v \rbrace \subseteq V(G)}d_G(u,v),$$
where $d_G(u,v)$ represents the distance between vertices $u$ and $v$ in $G$. It was introduced by H.\ Wiener in 1947 \cite{wiener} in order to calculate the boiling points of paraffins. Until now, it has found various applications in chemistry and in network theory.

It is easy to see that if $T$ is a tree, then
$$W(T) = \sum_{e =uv \in E(T)}n_u(e)n_v(e),$$
where $n_u(e)$ denotes the number of vertices of $T$ whose distance to $u$ is smaller than the distance to $v$ and $n_v(e)$ denotes the number of vertices of $T$ whose distance to $v$ is smaller than the distance to $u$. Therefore, I. Gutman \cite{gut_sz} used this formula to introduce the \textit{Szeged index}, which is for any connected graph $G$ defined as

$$Sz(G) = \sum_{e=uv \in E(G)}n_u(e)n_v(e).$$
Motivated by the success of the Szeged index, in \cite{def_pi} a similar molecular descriptor, that is called the \textit{PI index} (or the \textit{edge-PI index}), was defined with
$$PI(G)=PI_e(G) = \sum_{e=uv \in E(G)}\big(m_u(e) + m_v(e)\big),$$
where the numbers $m_u(e)$ and $m_v(e)$ are the edge-variants
of the numbers $n_u(e)$ and $n_v(e)$. It turned out that the Szeged index and the PI index also have many applications, for example in drug modelling \cite{drug} and in networks \cite{pisanski_rand}.

Later, the vertex version of the PI index, called the \textit{vertex-PI index}  \cite{khal1}, and the edge-version of the Szeged index, the \textit{edge-Szeged index} \cite{gut}, were defined as
$$PI_v(G) = \sum_{e=uv \in E(G)}\big(n_u(e) + n_v(e)\big), \quad Sz_e(G) = \sum_{e=uv \in E(G)}m_u(e)m_v(e).$$

Furthermore, in 2002 Randi\' c  introduced \cite{randic} the \textit{revised Szeged index}, $$Sz^*(G) = \sum_{e=uv \in E(G)}\left( n_u(e) + \frac{n_0(e)}{2} \right) \left( n_v(e) + \frac{n_0(e)}{2} \right),$$ where $n_0(e)$ is the number of vertices of the same distance from both $u$ and $v$ for $e =uv \in E(G)$. It turned out that the revised Szeged index has even better correlations with boiling points of cyclo-alkanes and that it can be used for measuring network bipartivity \cite{pisanski_rand}. Finally, the \textit{Mostar index} was recently introduced \cite{mostar} as a measure of peripherality and it is defined as $$Mo(G) = \sum_{e=uv \in E(G)} |n_u(e) - n_v(e)|.$$ 

In addition, Ili\'{c} and Milosavljevi\'{c} \cite{ilic} proposed  modifications of the Szeged index and the vertex-PI index, taking into account also the degrees of vertices. Therefore, they introduced the \textit{weighted Szeged index} and the \textit{weighted vertex-PI index}, which are defined as
\begin{eqnarray*}
wSz(G) &= &  \sum_{e=uv \in E(G)}\big(\deg(u) + \deg(v)\big)n_u(e)n_v(e),\\ wPI_v(G) &= &  \sum_{e=uv \in E(G)}\big(\deg(u) + \deg(v)\big)\big( n_u(e) + n_v(e)\big).
\end{eqnarray*}

Note that all the mentioned indices are often referred to as \textit{Szeged-like topological indices}. These quantities are some of the central and most commonly studied distance-based topological descriptors. For example, see recent research on Szeged indices \cite{kl-2020,bok,he,kl-2018}, Mostar indices \cite{deng,huang,imran,tepeh}, PI indices \cite{ma2}, and other versions of these indices \cite{ghorbani,li1,liu}.

A cut method has an important role in the investigation of molecular descriptors. Very often it was applied to benzenoid systems to efficiently compute distance-based topological indices, for example the Wiener index and the Szeged index \cite{chepoi-1997} or the edge-Szeged index and the PI index \cite{tratnik_edge_sz}. Later, a cut method was  generalized such that it can be used on partial cubes or on any connected graph by using $\Theta$ relation \cite{nad_klav,li}. When applying this method, we usually calculate a topological index by using weighted quotient graphs. A survey on different cut methods can be found in \cite{klavzar-2015}.

As already mentioned, a cut method was applied on various Szeged-like topological indices. For example, see recent papers on Szeged indices \cite{li}, weighted Szeged and PI indices \cite{tratnik_weighted_sz}, Mostar indices and weighted Mostar indices \cite{tratnik_mostar,aroc_cle_trat,aroc2}, and different distance-based topological indices \cite{aroc1,aroc3,aroc4}. In this paper, we greatly generalize these results 
by introducing the concept of a general Szeged-like topological index, which includes all the mentioned indices and also infinitely many other topological indices that can be defined in a similar way. We provide a cut method for computing a general Szeged-like topological index for any strength-weighted graph, which rounds off such investigations. 
Our method can be applied to efficiently calculate Szeged-like topological indices of various nanostructures and to deduce closed-form formulas for infinite families of graphs.

In Section 2 we introduce some basic definitions and concepts from graph theory. The concept of a strength-weighted graph and a general Szeged-like topological index is explained in the next section. We continue with the main result in Section 4, where a general cut method for computing a Szeged-like topological index of a strength-weighted graph is provided. Finally, in Section 5, applications of our cut method to well-known families of molecular graphs are shown. In particular, we consider benzenoid systems, phenylenes, and coronoid systems.

\section{Preliminaries}

The reader can find the explanation of all the basic concepts from graph theory in \cite{klavzar-book}. The graphs considered in this paper are simple and finite. For a graph $G$, the set of all the vertices is denoted by $V(G)$ and the set of edges by $E(G)$. Moreover, we define $d_G(u,v)$ to be the usual shortest-path distance between vertices $u, v\in V(G)$. Furthermore, the distance between a vertex $u \in V(G)$ and an edge $f=xy \in E(G)$ is defined as
$$d_G(u,f) = \min \lbrace d_G(u,x), d_G(u,y) \rbrace.$$

\noindent
For any $u \in V(G)$, the {\it open neighbourhood} $N(u)$ is defined as the set of all the vertices that are adjacent to $u$. The \textit{degree} of $u$, denoted by $\textrm{deg}(u)$, is defined as the cardinality of the set $N(u)$. 
\smallskip

\noindent
Two edges $e_1 = u_1 v_1$ and $e_2 = u_2 v_2$ of a connected graph $G$ are in relation $\Theta$, $e_1 \Theta e_2$, if
$$d(u_1,u_2) + d(v_1,v_2) \neq d(u_1,v_2) + d(u_2,v_1).$$
Note that this relation is also known as Djokovi\' c-Winkler relation.
The relation $\Theta$ is reflexive and symmetric, but not necessarily transitive.
We denote its transitive closure (i.e.\ the smallest transitive relation containing $\Theta$) by $\Theta^*$. The following easy observations will be useful:
\begin{itemize}
\item any two diametrically opposite edges in an even cycle are in relation $\Theta$,
\item any two edges in an odd cycle are in relation $\Theta^*$.
\end{itemize}
An important family of graphs, which is closely related to relation $\Theta$ and include many chemical graphs, are so-called partial cubes. Note that a connected graph is a partial cube if and only if it is bipartite and $\Theta = \Theta^*$. For the definition of a partial cube and other information on these graphs see \cite{klavzar-book}.
\smallskip

\noindent
Let $ \mathcal{F} = \lbrace F_1, \ldots, F_r \rbrace$ be the $\Theta^*$-partition of the set $E(G)$. Then we say that a partition $\mathcal{E} = \lbrace E_1, \ldots, E_k \rbrace$ of $E(G)$ is \textit{coarser} than $\mathcal{F}$
if each set $E_i$ is the union of one or more $\Theta^*$-classes of $G$. In such a case, $\mathcal{E}$ is also called a \textit{c-partition} of the set $E(G)$.
\bigskip

\noindent
Suppose $G$ is a graph and $F \subseteq E(G)$. The \textit{quotient graph} $G / F$ is a graph whose vertices are connected components of the graph $G \setminus F$, such that two components $X$ and $Y$ are adjacent in $G / F$ if some vertex in $X$ is
adjacent to a vertex of $Y$ in $G$. Note that $G \setminus F$ denotes the graph obtained from $G$ by removing all the edges in $F$. Moreover, if $E=XY$ is an edge in $G/F$, then we denote by $\widehat{E}$ the set of edges of $G$ that have one end vertex in $X$ and the other end vertex in $Y$, i.e. $\widehat{E}=  \lbrace xy \in E(G)\,|\,x \in V(X), y \in V(Y) \rbrace $. 
\bigskip

\noindent
Let $G$ be a connected graph, $\lbrace E_1, \ldots, E_k \rbrace$ a c-partition of the set $E(G)$, and $G/E_i$, $i \in \lbrace 1, \ldots, k \rbrace$, the corresponding quotient graph. We define the function $\ell_i: V(G) \rightarrow V(G/E_i)$ as follows: for any $u \in V(G)$, let $\ell_i(u)$ be the connected component $U$ of the graph $G \setminus E_i$ such that $u \in V(U)$.

\baselineskip=16pt

\section{General Szeged-like topological index of a strength-weighted graph}

\par The concept of a \textit{strength-weighted graph} was firstly introduced in \cite{aroc} as a triple $ G_{sw}= (G,SW_V,SW_E)$ where $ G $ is a simple graph and $ SW_V $, $ SW_E $ are pairs of weighted functions defined on $V(G)$ and $E(G)$, respectively:
\begin{itemize}
\item $ SW_V= (w_v,s_v)$, where $w_v , s_v : V(G_{sw}) \rightarrow {\mathbb R}_0^+$,
\item $ SW_E=(w_e,s_e)$, where $w_e, s_e: E(G_{sw}) \rightarrow {\mathbb R}_0^+$.
\end{itemize}


\noindent
For an edge $e=uv\in E(G)$ of a connected graph $G$, we define the following sets:
\begin{eqnarray*}
N_u(e|G) & = & \lbrace x \in V(G) \ | \ d_G(u,x) < d_G(v,x) \rbrace, \\
N_v(e|G) & = & \lbrace x \in V(G) \ | \ d_G(v,x) < d_G(u,x) \rbrace,\\
N_0(e|G) & = & \lbrace x \in V(G) \ | \ d_G(u,x) = d_G(v,x) \rbrace,\\
M_u(e|G) & = & \lbrace f \in E(G) \ | \ d_G(u,f) < d_G(v,f) \rbrace,\\
M_v(e|G) & = & \lbrace f \in E(G) \ | \ d_G(v,f) < d_G(u,f) \rbrace, \\
M_0(e|G) & = & \lbrace f \in E(G) \ | \ d_G(u,f) = d_G(v,f) \rbrace. 
\end{eqnarray*}

\noindent
Finally, for an edge $e=uv \in E(G_{sw})$ of a connected strength-weighted graph $G_{sw}$, we define the following quantities:
\begin{eqnarray*}
n_{u}(e\rvert G_{sw}) &= & \sum \limits_{x \in N_{u}(e \rvert G_{sw})}w_{v}(x), \\ 
m_{u}(e\rvert G_{sw}) &= &\sum \limits_{x \in N_{u}(e \rvert G_{sw})}s_{v}(x)+\sum \limits_{f \in M_{u}(e \rvert G_{sw})}s_{e}(f),\\
n_0(e\rvert G_{sw}) &=& \sum \limits_{x \in N_0(e \rvert G_{sw})}w_{v}(x),\\
m_0(e\rvert G_{sw}) &=& \sum \limits_{x \in N_0(e \rvert G_{sw})}s_{v}(x)+\sum \limits_{f \in M_0(e \rvert G_{sw})}s_{e}(f).
\end{eqnarray*}

\noindent Moreover, the values $n_{v}(e\rvert G_{sw})$ and $m_{v}(e\rvert G_{sw})$ are defined analogously.
\bigskip

To formally introduce a general Szeged-like topological index, the concept of a regular function of six variables is firstly needed.

\noindent
Let $X \subseteq \mathbb{R}^6$ and let $F: X \rightarrow \mathbb{R}$ be a function of six variables such that  $$F(x_1,x_2,x_3,x_4,x_5,x_6) = F(x_2,x_1,x_4,x_3,x_5,x_6)$$
 for all $(x_1,x_2,x_3,x_4,x_5,x_6) \in X$. With other words, $F$ is symmetric in the first two coordinates and in the next two coordinates.  Moreover, for any edge $e=uv$ of a connected strength-weighted graph $G_{sw}$ we introduce the following notation:
$$F(e|G_{sw}) = F \big(n_u(e|G_{sw}), n_v(e|G_{sw}), m_u(e|G_{sw}), m_v(e|G_{sw}), n_0(e|G_{sw}), m_0(e|G_{sw}) \big).$$
We always assume that the number $F(e|G_{sw})$ is well defined for any edge $e \in E(G)$. 
A function $F$ satisfying the mentioned requirements will be called a \textit{regular function} for a graph $G_{sw}$. We remark that $F$ should be symmetric because any edge $e=uv \in E(G_{sw})$ can be also written as $e=vu$.

Now everything is prepared to define the general Szeged-like topological index. 
\begin{definition}
If $F$ is a regular function for a strength-weighted connected graph $G_{sw}$, then the \textbf{Szeged-like topological index} of $G_{sw}$, denoted by $TI_F(G_{sw})$,  is defined as
$$TI_F(G_{sw}) = \sum_{e \in E(G)} w_e(e) F(e|G_{sw}).$$
\end{definition}

Obviously, many well-known distance-based topological indices are just special cases of the general Szeged-like topological index. To show this, let $G$ be a connected graph. We obtain the strength-weighted graph $G_{sw}$ in the following way:  we set $w_{v} \equiv 1$, $s_{e} \equiv 1$, and $s_{v} \equiv 0$. Moreover,  weight $w_e(e)$, where $e=uv \in E(G)$, function $F$, and the corresponding topological index are shown in Table \ref{tabela1}.

\begin{table}[h]
\begin{center}
\begin{tabular}{|c|c|c|c|c|c|}
 \hline 
\textbf{ topological index} & \textbf{regular function} $F$ & $w_e (uv)$  \\ 
 \hline 
 Szeged index ($Sz$) & $x_1x_2$ & 1  \\ 
 \hline 
 edge-Szeged index ($Sz_e$)& $x_3x_4$ & 1  \\ 
  \hline 
revised Szeged index ($Sz^*$)& $(x_1 + x_5/2)(x_2 + x_5/2)$ & 1  \\ 
  \hline
revised edge-Szeged index ($Sz_e^*$)& $(x_3 + x_6/2)(x_4 + x_6/2)$ & 1  \\ 
 \hline 
 vertex-edge Szeged index ($Sz_{ve}$) & $x_1x_4 + x_2x_3$ & 1 \\ 
 \hline 
 total Szeged index ($Sz_{t}$) & $(x_1 + x_3)(x_2 + x_4)$ & 1  \\ 
  \hline 
weighted-plus Szeged index ($w^+Sz$) & $x_1x_2$ & $\deg (u) + \deg (v)$ \\ 
  \hline 
weighted-product Szeged index ($w^*Sz$) & $x_1x_2$ & $\deg (u)  \deg (v)$ \\ 
 \hline 
weighted-plus edge-Sz.\,index ($w^+Sz_e$) & $x_3x_4$ & $\deg (u) + \deg (v)$ \\ 
  \hline 
weighted-prod.\,edge-Sz.\,index ($w^*Sz_e$) & $x_3x_4$ & $\deg (u)  \deg (v)$ \\ 
 \hline 
weighted-plus total-Sz.\,index ($w^+Sz_t$) & $(x_1 + x_3)(x_2 + x_4)$ & $\deg (u) + \deg (v)$ \\ 
  \hline 
weighted-prod.\,total-Sz.\,index ($w^*Sz_t$) & $(x_1 + x_3)(x_2 + x_4)$ & $\deg (u)  \deg (v)$ \\ 
 \hline 
 (edge-)PI index ($PI$) & $x_3 + x_4$ & 1  \\ 
 \hline 
 vertex-PI index ($PI_v$)& $x_1 + x_2$ & 1  \\ 
 
 \hline 
 total PI index ($PI_{t}$) & $x_1 + x_2 + x_3 + x_4$ & 1  \\ 
  \hline 
weighted-plus PI index ($w^+PI$) & $x_3 + x_4$ & $\deg (u) + \deg (v)$ \\ 
  \hline 
weighted-product PI index ($w^*PI$) & $x_3 + x_4$ & $\deg (u)  \deg (v)$ \\ 
 \hline 
weighted-plus vertex-PI index ($w^+PI_v$) & $x_1 + x_2$ & $\deg (u) + \deg (v)$ \\ 
  \hline 
weighted-prod.\,vertex-PI index ($w^*PI_v$) & $x_1 + x_2$ & $\deg (u)  \deg (v)$ \\ 
 \hline 
Mostar index ($Mo$) & $|x_1 - x_2|$ & 1  \\ 
 \hline 
 edge-Mostar index ($Mo_e$)& $|x_3 - x_4|$ & 1  \\ 
 \hline 
 total Mostar index ($Mo_{t}$) & $|x_1 + x_3 - x_2 - x_4|$ & 1  \\ 
  \hline 
weighted-plus Mostar index ($w^+Mo$) & $|x_1 - x_2|$ & $\deg (u) + \deg (v)$ \\ 
  \hline 
weighted-product Mostar index ($w^*Mo$) & $|x_1 - x_2|$ & $\deg (u)  \deg (v)$ \\ 
 \hline 
weighted-plus edge-Mo.\,index ($w^+Mo_e$) & $|x_3 - x_4|$ & $\deg (u) + \deg (v)$ \\ 
  \hline 
weighted-prod.\,edge-Mo.\,index ($w^*Mo_e$) & $|x_3 - x_4|$ & $\deg (u)  \deg (v)$ \\ 
  \hline 
weighted-plus total-Mo.\,index ($w^+Mo_t$) & $|x_1 + x_3 - x_2 - x_4|$ & $\deg (u) + \deg (v)$ \\ 
  \hline 
weighted-prod.\,total-Mo.\,index ($w^*Mo_t$) & $|x_1 + x_3 - x_2 - x_4|$ & $\deg (u)  \deg (v)$ \\ 
 \hline 
 \end{tabular} 
 \end{center} 
 \caption{ \label{tabela1} Topological indices, corresponding functions, and weights $w_e$.}
 \end{table}
 
 Based on the above discussion and Table \ref{tabela1}, we say that a strength-weighted graph $G_{sw}$ is \textit{normally strength-weighted}, if $w_{v} \equiv 1$, $s_{e} \equiv 1$, $s_{v} \equiv 0$, and for $w_e$ we have one of the following options:
 \begin{itemize}
 \item [$(i)$] $w_e(e) \equiv 1$,
 \item [$(ii)$] $w_e(e) = \deg (u) + \deg (v)$ for any $e=uv$ (in this case, we often use $w_e^+(e)$),
 \item [$(iii)$] $w_e(e) = \deg (u) \deg (v)$ for any $e=uv$ (in this case, we often use $w_e^*(e)$).
 \end{itemize}
 
 However, our general definition of the Szeged-like topological index includes also infinitely many other topological indices, which are more complicated and are not presented in Table \ref{tabela1}. To show this, we will, as an example, consider the \textit{weighted-plus revised edge-Szeged index},
 $$w^+Sz_e^*(G) = \sum_{e =uv \in E(G)} \left( \deg(u) + \deg(v) \right)\left( m_u + \frac{m_0}{2} \right) \left( m_v + \frac{m_0}{2} \right),$$
and the \textit{square vertex-PI index},
 
 $$PI_v^s(G) = \sum_{e =uv \in E(G)} \left( n_u^2 + n_v^2 \right).$$
 
\noindent
Obviously, for the first index we assume that $F=(x_3 + x_6/2)(x_4 + x_6/2)$ and $w_e(uv) = \deg (u) + \deg (v)$, while for the second index we have $F= x_1^2 + x_2^2$ and $w_e \equiv 1$.

 
  In the rest of the section, we show that the Szeged-like topological index of a tree can be computed in linear time. Let $F$ be a regular function for a graph $G_{sw}$. In the next lemma and in all other computational results of this paper, we always assume that for every edge $e \in E(G_{sw})$, the number $F(e|G_{sw})$ can be evaluated in constant time $O(1)$. To emphasis that we work in such a model, any function $F$ satisfying this condition will be called a \textit{normal function} for  graph $G_{sw}$. The mentioned assumption is usually made if $F$ can be expressed by a fixed number of basic arithmetic and logic operations (for example all functions from Table \ref{tabela1}).
 
 \begin{lemma} \label{trees_linear1}
Let $T_{sw}$ be a  strength-weighted tree with $n$ vertices. If $F$ is a normal function for $T_{sw}$, then the index $TI_F(T_{sw})$ can be computed in $O(n)$ time.
\end{lemma}

\begin{proof}
Since $F$ is a normal function, the proof can be done in a similar way as in Lemma 4.1 from \cite{tratnik_edge_sz}. The only difference is that we should consider four weights instead of just two. \qed 
\end{proof}

\section{A general cut method}

In this section, we prove a method for computing any Szeged-like topological index of a connected strength-weighted graph from the corresponding quotient graphs. In this way, we greatly generalize many previous results.

Let $G_{sw}$ be a connected strength-weighted graph and $ \{E_{1},\ldots,E_{k}\}$ be a c-partition of $E(G)$. Moreover, for $i \in \lbrace 1, \ldots, k \rbrace$ let $G_{sw}/E_i=(G/E_i,SW_v^i, SW_e^i) $ be the strength-weighted quotient graph, where the weights $w_v^i$, $s_v^i$, $w_e^i$, and $s_e^i$ are defined as follows \cite{aroc,aroc2}:  
\begin{itemize}
	\item $w_{v}^{i} : V(G_{sw}/E_{i}) \rightarrow \mathbb{R}_0^{+}$,\ $w_{v}^{i}(X)  = \sum \limits_{x \in V(X)} w_{v}(x)$, $\forall$ $X \in V(G_{sw}/E_{i})$,
	\item $s_{v}^{i} : V(G_{sw}/E_{i}) \rightarrow \mathbb{R}_0^{+}$,\ $s_{v}^{i}(X) = \sum \limits_{f \in E(X)} s_{e}(f) \ +\sum \limits_{x \in V(X)} s_{v}(x) $, $\forall$ $X \in V(G_{sw}/E_{i})$,
	\item $w_{e}^{i} : E(G_{sw}/E_{i}) \rightarrow {\mathbb R}_0^+$,\
			  $w_{e}^{i}(E)= \sum \limits_{e \in \widehat{E}}w_{e}(e)$, $ \forall \ E \in E(G_{sw}/E_{i})$,
	\item $s_{e}^{i} : E(G_{sw}/E_{i}) \rightarrow \mathbb{R}_0^{+}$,\ $s_{e}^{i}(E) = \sum \limits_{e \in \widehat{E}}s_{e}(e), \forall \ E \in E(G_{sw}/E_{i})$.
\end{itemize}

Moreover, it is easy to see that if $G_{sw}$ is a normally strength-weighted graph, then the following holds:
\begin{itemize}
	\item  $w_{v}^{i}(X)$ is the number of vertices in a connected component  $X$ of $G_{sw} \setminus E_{i}$,
	\item $s_{v}^{i}(X)$ is the number of edges in a connected component  $X$ of $G_{sw} \setminus E_{i}$,
	
	\item $s_{e}^{i}(E)= \left| \widehat{E} \right|$ for $E \in E(G_{sw}/E_{i})$. In other words, if $E=XY$, then $s_{e}^{i}(E)$ is the number of edges between connected components $X$ and $Y$ of $G_{sw} \setminus E_{i}$.
	\item for $w_e^i(E)$, where $E \in E(G_{sw}/E_{i})$, we have one of the following cases:
	\begin{itemize}
	\item [$(i)$] if $w_e \equiv 1$, then $w_e^i(E) = s_{e}^{i}(E)$,
	\item [$(ii)$] if $w_e=w_e^+$, then $ w_e^i(E) = \displaystyle\sum_{e=uv \in \widehat{E}} \left( \deg(u) + \deg(v) \right)$,
	\item [$(iii)$] if $w_e = w_e^*$, then $w_e^i(E) = \displaystyle\sum_{e=uv \in \widehat{E}}  \deg(u) \deg(v) $.
	\end{itemize}
\end{itemize}

\par Throughout the paper, the quotient graph $G_{sw}/E_i$ will be shortly denoted as $G_i$ for any $i \in \lbrace 1, \ldots, k \rbrace$.
Firstly, we need two lemmas.

\begin{lemma} \cite{tratnik_mostar,tratnik_weighted_sz} \label{pomoc0}
Let $G$ be a connected graph. If $e =uv \in E_i$, where $i \in \lbrace 1, \ldots, k \rbrace$, then $U=\ell_i(u)$ and $V=\ell_i(v)$ are adjacent vertices in $G_i$, i.e.\ $E=UV \in E(G_i)$. Moreover,

\begin{eqnarray*}
&N_u(e|G) &= \bigcup_{X \in N_{U}(E|G_i)} V(X), \\
 &N_v(e|G) &= \bigcup_{X \in N_{V}(E|G_i)} V(X),\\
&M_u(e|G) &= \left( \bigcup_{X \in N_{U}(E|G_i)} E(X) \right) \bigcup \left( \bigcup_{F \in M_{U}(E|G_i)} \widehat{F} \right),\\
&M_v(e|G) &= \left( \bigcup_{X \in N_{V}(E|G_i)} E(X) \right) \bigcup \left( \bigcup_{F \in M_{V}(E|G_i)} \widehat{F} \right).
\end{eqnarray*}
\end{lemma}

\begin{lemma} \label{pomoc} \cite{aroc_cle_trat} Let $G_{sw}$ be a connected strength-weighted graph. If $e=uv \in E_i$, where $i \in \lbrace 1, \ldots, k \rbrace$, $U=\ell_i(u)$, $V=\ell_i(v)$, and $E= UV \in E(G_i)$, then \\[1em]
\begin{tabular}{l l l}
\textnormal{(\textit{i})} $n_u(e|G_{sw})= n_U(E|G_i)$&\textit{ and  }&$n_v(e|G_{sw})= n_V(E|G_i)$,\\[1em]
\textnormal{(\textit{ii})} $m_u(e|G_{sw})= m_U(E|G_i)$&\textit{ and  }&$m_v(e|G_{sw})= m_V(E|G_i)$.
\end{tabular}\smallskip
\end{lemma}

\noindent
The following lemma will be also useful.

\begin{lemma} \label{pomoc1} Let $G_{sw}$ be a connected strength-weighted graph. If $e=uv \in E_i$, where $i \in \lbrace 1, \ldots, k \rbrace$, $U=\ell_i(u)$, $V=\ell_i(v)$, and $E= UV \in E(G_i)$, then

$$ n_0(e|G_{sw}) = n_0(E|G_i) \ \textit{  and  } \ m_0(e|G_{sw}) = m_0(E|G_i). $$
\end{lemma}

\begin{proof} By Lemma \ref{pomoc0} we have
$$N_u(e|G) = \left( \bigcup_{X \in N_{U}(E|G_i)} V(X) \right) \textrm{ and } N_v(e|G) = \left( \bigcup_{X \in N_{V}(E|G_i)} V(X) \right).$$
Obviously, the set $\lbrace N_u(e|G), N_v(e|G), N_0(e|G) \rbrace$ is a partition of the set $V(G)$. Moreover, the set $\lbrace N_U(E|G_i), N_V(E|G_i), N_0(E|G_i) \rbrace$ is a partition of the set $V(G_i)$. Therefore, by the above formulas it follows
\begin{equation} \label{enacba1} N_0(e|G) = \left( \bigcup_{X \in N_{0}(E|G_i)} V(X) \right).
\end{equation}

\noindent
Hence, by Equation \eqref{enacba1} we can calculate
\begin{eqnarray*}
n_0(e|G) & = & \sum_{x \in N_0(e|G_{sw})} w_v(x) \\
& = & \sum_{X \in N_0(E|G_i)} \left ( \sum_{x \in V(X)} w_v(x) \right) \\
& = & \sum_{X \in N_0(E|G_i)} w_v^i(X) \\
& = & n_0(E|G_i).
\end{eqnarray*}

Similarly, by Lemma \ref{pomoc0} it holds
$$M_u(e|G) = \left( \bigcup_{X \in N_{U}(E|G_i)} E(X) \right) \bigcup \left( \bigcup_{F \in M_{U}(E|G_i)} \widehat{F} \right),$$
$$M_v(e|G) = \left( \bigcup_{X \in N_{V}(E|G_i)} E(X) \right) \bigcup \left( \bigcup_{F \in M_{V}(E|G_i)} \widehat{F} \right).$$

\noindent
Obviously, the set $\lbrace M_u(e|G), M_v(e|G), M_0(e|G) \rbrace$ is a partition of the set $E(G)$. Moreover, the sets $\lbrace N_U(E|G_i), N_V(E|G_i), N_0(E|G_i) \rbrace$ and $\lbrace M_U(E|G_i), M_V(E|G_i), M_0(E|G_i) \rbrace$ are partitions of the sets $V(G_i)$ and $E(G_i)$, respectively. Therefore, by the above formulas it follows
\begin{equation} \label{enacba2}
M_0(e|G) = \left( \bigcup_{X \in N_{0}(E|G_i)} E(X) \right) \bigcup \left( \bigcup_{F \in M_{0}(E|G_i)} \widehat{F} \right).
\end{equation}

\noindent
Hence, by Equations \eqref{enacba1} and \eqref{enacba2} we can calculate
\begin{eqnarray*}
m_0(e|G_{sw}) & = & \sum_{x \in N_0(e|G_{sw})} s_v(x) + \sum_{f \in M_0(e|G_{sw})} s_e(f) \\
& = & \sum_{X \in N_0(E|G_i)} \left ( \sum_{x \in V(X)} s_v(x) \right) \\
& + & \sum_{X \in N_0(E|G_i)} \left ( \sum_{f \in E(X)} s_e(f) \right)  + \sum_{F \in M_0(E|G_i)} \left ( \sum_{f \in \widehat{F}} s_e(f) \right)  \\
& = & \sum_{X \in N_0(E|G_i)} \left ( \sum_{x \in V(X)} s_v(x) + \sum_{f \in E(X)} s_e(f) \right) \\
& + & \sum_{F \in M_0(E|G_i)} \left ( \sum_{f \in \widehat{F}} s_e(f) \right)  \\
& = & \sum_{X \in N_0(E|G_i)} s_v^i(X) + \sum_{F \in M_0(E|G_i)} s_e^i(F)\\
& = & m_0(E|G_i),
\end{eqnarray*}
which completes the proof. \qed
\end{proof}

\bigskip

\noindent
Based on the obtained results, the next lemma follows easily.

\begin{lemma} \label{pomoc3}
Let $G_{sw}$ be a connected strength-weighted graph and $F$ a regular function for $G_{sw}$. If $e=uv \in E_i$, where $i \in \lbrace 1, \ldots, k \rbrace$, $U=\ell_i(u)$, $V=\ell_i(v)$, and $E= UV \in E(G_i)$, then
$$F(e|G_{sw}) = F(E|G_i).$$
\end{lemma}

\begin{proof}
The proof follows by Lemma \ref{pomoc} and Lemma \ref{pomoc1}. \qed
\end{proof}

The main theorem of this paper can now be stated.

\begin{theorem} \label{glavni} Let $G_{sw}$ be a connected strength-weighted graph. If $\lbrace E_1, \ldots, E_k \rbrace$ is a c-partition of $E(G_{sw})$ and $F$ a regular function for $G_{sw}$, then
$$TI_F(G_{sw}) = \sum_{i=1}^k TI_F(G_i).$$
\end{theorem}

\begin{proof}
Obviously, $E(G) = \displaystyle\bigcup_{i=1}^k E_i$ and for any $i \in \lbrace 1,\ldots, k \rbrace$ it holds $$E_i= \bigcup_{E \in E(G_i)} \widehat{E}.$$
Then, by Lemma \ref{pomoc3} we get
\begin{eqnarray*}
TI_F(G_{sw}) & = & \sum_{e \in E(G_{sw})}w_{e}(e)  F(e|G_{sw})  \\[1em]
& = & \sum_{i=1}^k  \Bigg( \sum_{e \in E_{i}}w_{e}(e) F(e|G_{sw}) \Bigg) \\[1em]
  & = & \sum_{i=1}^k \Bigg( \sum_{E \in E(G_i)} \Bigg[ \sum_{e  \in \widehat{E}} w_{e}(e) F(e|G_{sw}) \Bigg] \Bigg)\\[1em]
    & = & \sum_{i=1}^k \Bigg( \sum_{E \in E(G_i)} \Bigg[ \sum_{e  \in \widehat{E}} w_{e}(e)F(E|G_i) \Bigg] \Bigg)\\[1em]
     & = & \sum_{i=1}^k \Bigg( \sum_{E \in E(G_i)} \Bigg[ \sum_{e  \in \widehat{E}} w_{e}(e) \Bigg] F(E|G_i)  \Bigg)\\[1em]
          & = & \sum_{i=1}^k \Bigg( \sum_{E \in E(G_i)} w_{e}^{i}(E) F(E|G_i)  \Bigg)\\[1em]
& = & \sum_{i=1}^k TI_F(G_i). 
\end{eqnarray*}
Therefore, the proof is complete. \qed
\end{proof}

\section{Applications to molecular graphs}

We apply our main result to some important families of molecular graphs. In particular, benzenoid systems, phenylenes and coronoid systems are considered.

\subsection{Benzenoid systems}

In this subsection, we show how Theorem \ref{glavni} can be applied to benzenoid systems. These chemical graphs represent benzenoid hydrocarbons, which are composed exclusively of six-membered rings. For more information, see \cite{gucy-89}. It is well known that benzenoid systems are partial cubes \cite{klavzar-book}.

\noindent
Let ${\cal H}$ be the infinite hexagonal (graphite) lattice and let $Z$ be a cycle on it. A {\em benzenoid system} is the graph induced by all the vertices and edges of ${\cal H}$, lying on $Z$ or in its interior. In addition, by $|Z|$ we denote the number of vertices in $Z$. For an example of a benzenoid system, see Figure \ref{ben_primer}. 

\begin{figure}[h!] 
\begin{center}
\includegraphics[scale=0.7,trim=0cm 0.4cm 0cm 0cm]{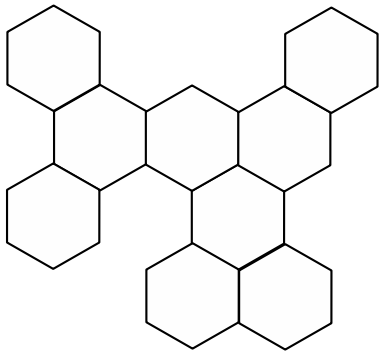}
\end{center}
\caption{\label{ben_primer} Benzenoid system $G$.}
\end{figure}

It turns out that any $\Theta$-class of a benzenoid system
has a nice geometric representation, since it coincides with exactly one of its elementary cuts (see \cite{klavzar-book}). An \textit{elementary cut} of a benzenoid system $G$ is a line segment that starts at the center of a peripheral edge of a benzenoid system,
goes orthogonal to it and ends at the first next peripheral
edge of $G$. 

The edge set of a benzenoid system $G$ can be naturally partitioned into sets $E_1, E_2$ and $E_3$ of edges of the same direction. Obviously, the partition $\lbrace E_1,E_2,E_3 \rbrace$ is a c-partition of the set $E(G)$. If $G_{sw}$ is a strength-weighted benzenoid system, then for $i \in \lbrace 1, 2, 3 \rbrace$, the strength-weighted quotient graph $G_i=G/E_i$ will be denoted as $T_i$. It is well known that $T_1$, $T_2$, and $T_3$ are trees \cite{chepoi-1996}. Such quotient trees were previously used to calculate various distance-based topological indices of benzenoid systems, for example see \cite{chepoi-1997,tratnik_edge_sz,tratnik_weighted_sz}. 

Let $G$ be a benzenoid system from Figure \ref{ben_primer}. Moreover, let $G_{sw}$ be a normally strength-weighted benzenoid system obtained from $G$ such that $w_e(e)= \deg(u) + \deg(v)$ for any $e \in E(G_{sw})$. Then, let $E_1$ be the set of all the vertical edges of $G$. Therefore, the edges from $E_1$ correspond to horizontal elementary cuts of $G_{sw}$, see Figure \ref{quotient_tree1} a). The corresponding strength-weighted quotient tree is shown in Figure \ref{quotient_tree1} b). Similarly, we obtain also the other two quotient trees, see Figure \ref{quotient_tree2}.

\begin{figure}[h!] 
\begin{center}
\includegraphics[scale=0.7,trim=0cm 0.4cm 0cm 0cm]{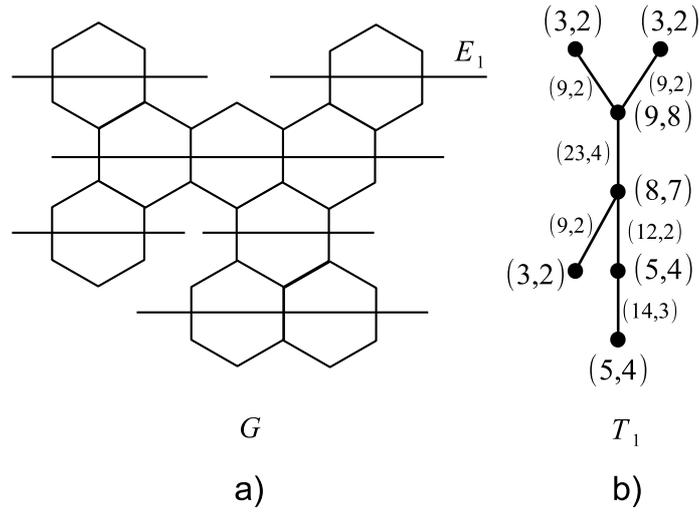}
\end{center}
\caption{\label{quotient_tree1} a) Elementary cuts in $E_1$ and b) the corresponding strength-weighted quotient tree $T_1$.}
\end{figure}

\begin{figure}[h!] 
\begin{center}
\includegraphics[scale=0.7,trim=0cm 0.4cm 0cm 0cm]{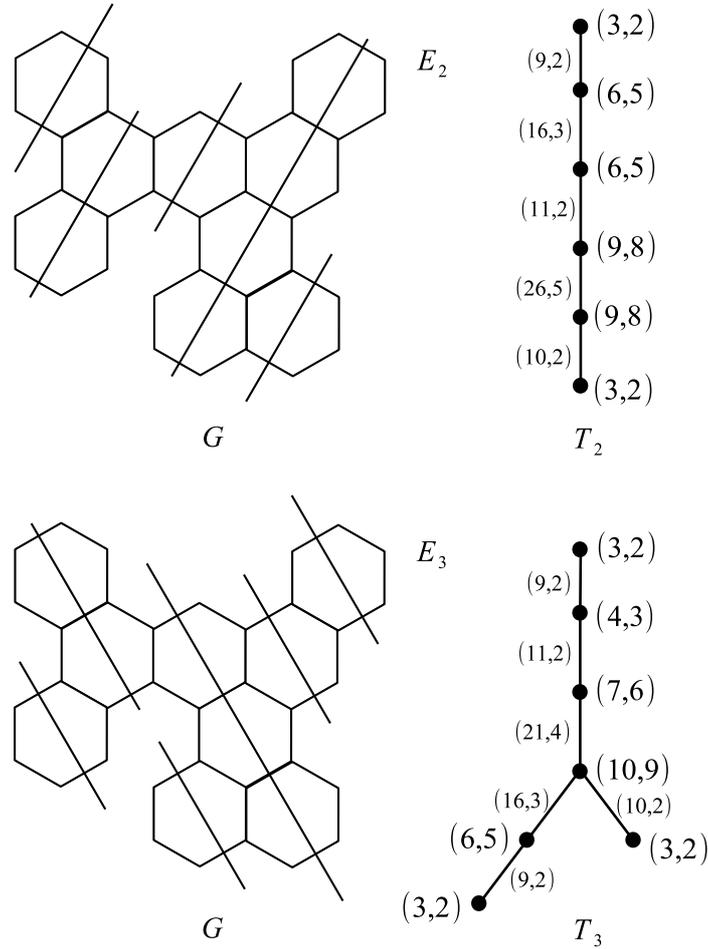}
\end{center}
\caption{\label{quotient_tree2} Elementary cuts in a given direction and the corresponding strength-weighted quotient trees.}
\end{figure}

The following proposition follows directly from Theorem \ref{glavni}. 

\begin{proposition} \label{ben} Let $G_{sw}$ be a strength-weighted benzenoid system. If $F$ is a regular function for $G_{sw}$ and $T_1, T_2, T_3$ are the corresponding strength-weighted quotient trees, then
$$TI_F(G_{sw}) = TI_F(T_1) + TI_F(T_2) + TI_F(T_3).$$
\end{proposition}

To finish the previous example, we compute the weighted-plus revised edge-Szeged index and the square vertex-PI index of graph $G$ from Figure \ref{ben_primer} by using Proposition \ref{ben}. Firstly, we calculate the corresponding indices of strength-weighted quotient trees shown in Figure \ref{quotient_tree1} and Figure \ref{quotient_tree2}. For the weighted-plus revised edge-Szeged index we obtain
\begin{eqnarray*}
w^+Sz_e^*(T_1) &=& 21657.5, \\
w^+Sz_e^*(T_2) &=& 24234.5, \\
w^+Sz_e^*(T_3) &=& 21879, \\
w^+Sz_e^*(G) &=& w^+Sz_e^*(T_1) + w^+Sz_e^*(T_2) + w^+Sz_e^*(T_3)= 67771.
\end{eqnarray*}

Note that for the square vertex-PI index, we should have $w_e(e)=1$ for any $e \in E(G_{sw})$. Therefore, $w_e(e)=s_e(e)$ for $e \in E(G_{sw})$ and $w_e^i=s_e^i$ for $i \in \lbrace 1,2,3 \rbrace$. As a consequence, we get
\begin{eqnarray*}
PI_v^s(T_1) &=& 13762, \\
PI_v^s(T_2) &=& 11754, \\
PI_v^s(T_3) &=& 13518, \\
PI_v^s(G) &=& PI_v^s(T_1) +PI_v^s(T_2) + PI_v^s(T_3)= 39034.
\end{eqnarray*}

It can be shown that by Proposition \ref{ben}, a general Szeged-like topological index of a benzenoid system can we computed in linear time. Analogous results are already known for some topological indices \cite{chepoi-1997,tratnik_edge_sz}.

\begin{proposition}
Let $G_{sw}$ be a strength-weighted benzenoid system with $n$ vertices. If $F$ is a normal function for $G_{sw}$, then the index $TI_F(G_{sw})$ can be computed in $O(n)$ time.
\end{proposition}

\begin{proof}
It is already known that quotient trees $T_i$, $i \in \lbrace 1,2,3 \rbrace$, can be computed in linear time, see \cite{chepoi-1996}. The calculation of the corresponding  weights is also straightforward. By Lemma \ref{trees_linear1}, the index $TI_F(T_i)$ can be computed in linear time for any $i \in \lbrace 1,2,3 \rbrace$. The result now follows by Proposition \ref{ben}. \qed
\end{proof}

However, for normally strength-weighted benzenoid systems the calculation can be done even faster, in sublinear time. To show this, we follow the idea of paper \cite{CK-1998}, where analogous result was proved for the Wiener index. The next lemma will be  needed.

\begin{lemma}  \label{trees_linear}
If $G_{sw}$ is a normally strength-weighted benzenoid system and $Z$ its boundary cycle, then each strength-weighted quotient tree $T_i$, $i \in \lbrace 1,2,3 \rbrace$, can be obtained in $O(|Z|)$ time.	
\end{lemma}

\begin{proof}
The proof uses a special construction of strength-weighted quotient trees and it is similar to the proof of Lemma 4.3 in \cite{tratnik_weighted_sz}. It relies on Chazelle's algorithm \cite{chazelle} for computing all vertex-edge visible pairs of edges of a simple (finite) polygon in linear time.  Hence, the details are omitted. \qed
\end{proof}

\noindent
We can now state the final result of this section.

\begin{theorem}
Let $G_{sw}$ be a normally strength-weighted benzenoid system and $Z$ its boundary cycle. If $F$ is a normal function for $G_{sw}$, then the index $TI_F(G_{sw})$ can be computed in $O(|Z|)$ time.
\end{theorem}

\begin{proof}
By Lemma \ref{trees_linear} it follows that the strength-weighted trees $T_i$, $i \in \lbrace 1,2,3 \rbrace$, can be computed in $O(|Z|)$ time. Furthermore, by Lemma \ref{trees_linear1} the index $TI_F(T_i)$ of each tree can be computed in linear time with respect to $|Z|$. Finally,  $TI_F(G_{sw})$ can be computed in $O(|Z|)$ time by Proposition \ref{ben}.  \qed
\end{proof}

\subsection{Phenylenes}

Phenylenes are polycyclic conjugated molecules composed of hexagonal and quadrilateral cycles. Molecular descriptors for these molecules were investigated in many papers, see \cite{tratnik_weighted_sz,zigert-2018} as an example. In this subsection, we describe an efficient method for calculating Szeged-like topological indices of phenylenes. We also show that our method can be used to easily obtain closed-form formulas for such indices.

Next, we formally define a phenylene in the language of graph theory. A benzenoid system is said to be \textit{catacondensed} if all its vertices belong to the outer face. Moreover, two distinct hexagons of a benzenoid system are called \textit{adjacent} if they have exactly one edge in common. Let $G'$ be a catacondensed benzenoid system. If we add squares between all pairs of adjacent hexagons of $G'$, the obtained graph $G$ is called a \textit{phenylene}. We then say that $G'$ is the \textit{hexagonal squeeze} of $G$.

In the following, we define four quotient trees of a phenylene  \cite{zigert-2018}. Let $G$ be a phenylene and $G'$ the hexagonal squeeze of $G$. The edge set of $G'$ can be naturally partitioned into sets $E_1'$, $E_2'$ and $E_3'$ of edges of the same direction. Denote the sets of edges of $G$ corresponding to the edges in $E_1'$, $E_2'$ and $E_3'$ by $E_1, E_2$ and $E_3$, respectively. Moreover, let $E_4 = E(G) \setminus (E_1 \cup E_2 \cup E_3)$ be the set of all the edges of $G$ that do not belong to $G'$. Again we can easily see that phenylenes are partial cubes and that the partition $\lbrace E_1,E_2,E_3,E_4 \rbrace$ is a c-partition of the edge set $E(G)$. For $i \in \lbrace 1, 2, 3, 4 \rbrace$, set $T_i = G /E_i$. As in the previous section, we can see that $T_1$, $T_2$, $T_3$, and $T_4$ are trees. In a similar way we can define the quotient trees $T_1', T_2', T_3'$ of the hexagonal squeeze $G'$. Then the tree $T_i'$ is isomorphic to $T_i$ for $i=1,2,3$ and $T_4$ is isomorphic to the inner dual of $G'$ (see \cite{zigert-2018} for the definition of the inner dual). Finally, if $G_{sw}$ is a strength-weighted phenylene, then the corresponding strength-weighted trees will be also denoted by $T_i$ for $i \in \lbrace 1,2,3,4 \rbrace$.

The following proposition follows by Theorem \ref{glavni}.

\begin{proposition} \label{izr_phe} Let $G_{sw}$ be a strength-weighted phenylene. If $F$ is a regular function for $G_{sw}$ and $T_1, T_2, T_3, T_4$ are the corresponding strength-weighted quotient trees, then
$$TI_F(G_{sw}) = TI_F(T_1) + TI_F(T_2) + TI_F(T_3) + TI_F(T_4).$$
\end{proposition}

As an example, we consider an infinite family of phenylenes which will be denoted by $Ph_n$, $n \geq 1$ (see Figure \ref{phenylenes}). In the following, we deduce closed-form formulas for the weighted-plus revised edge-Szeged index and the square vertex-PI index of $Ph_n$. For this purpose, we assume that the phenylenes are normally strength-weighted such that $w_e = w_e^+$. The corresponding strength-weighted quotient trees are depicted in Figure \ref{phenylenes_quotient}. However, as in the previous subsection, for any $i \in \lbrace 1,2,3,4 \rbrace$ we use the weights $s_e^i$ of $T_i$ instead of $w_e^i$ in the computation of the square vertex-PI index.

Firstly, the weighted-plus revised edge-Szeged index of the quotient trees is calculated:
\begin{figure}[h]
\begin{center}
\includegraphics[scale=0.8,trim=0cm 0.4cm 0cm 0cm]{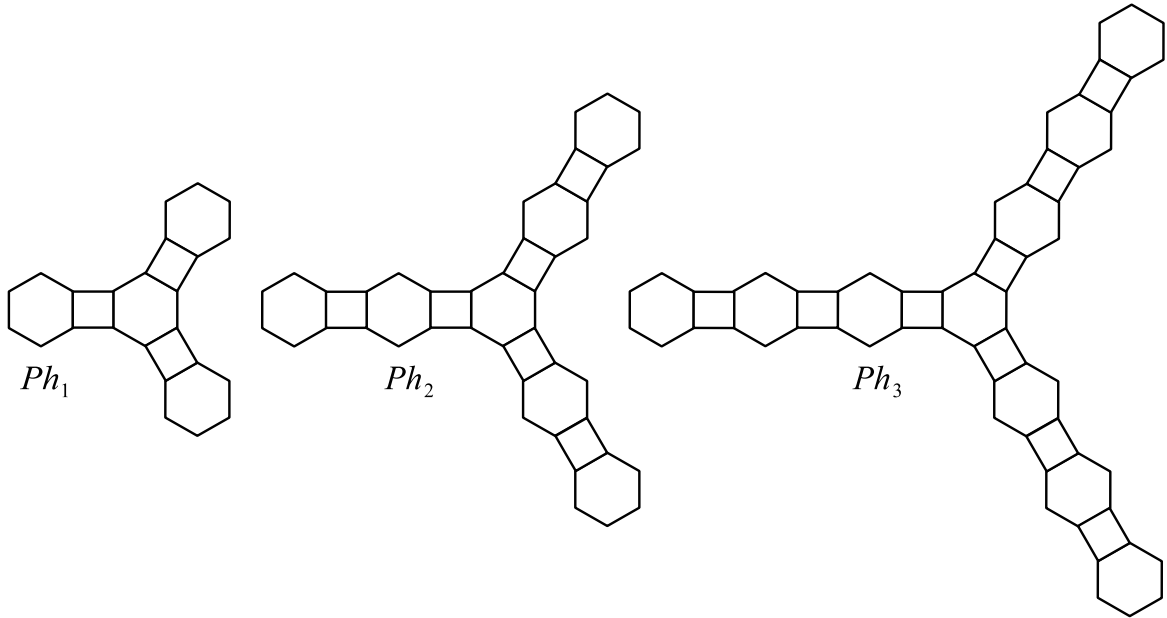}
\end{center}
\caption{\label{phenylenes} The first three representatives of a family of phenylenes $Ph_n$.}
\end{figure}

\begin{figure}[h] 
\begin{center}
\includegraphics[scale=0.7,trim=0cm 0.4cm 0cm 0cm]{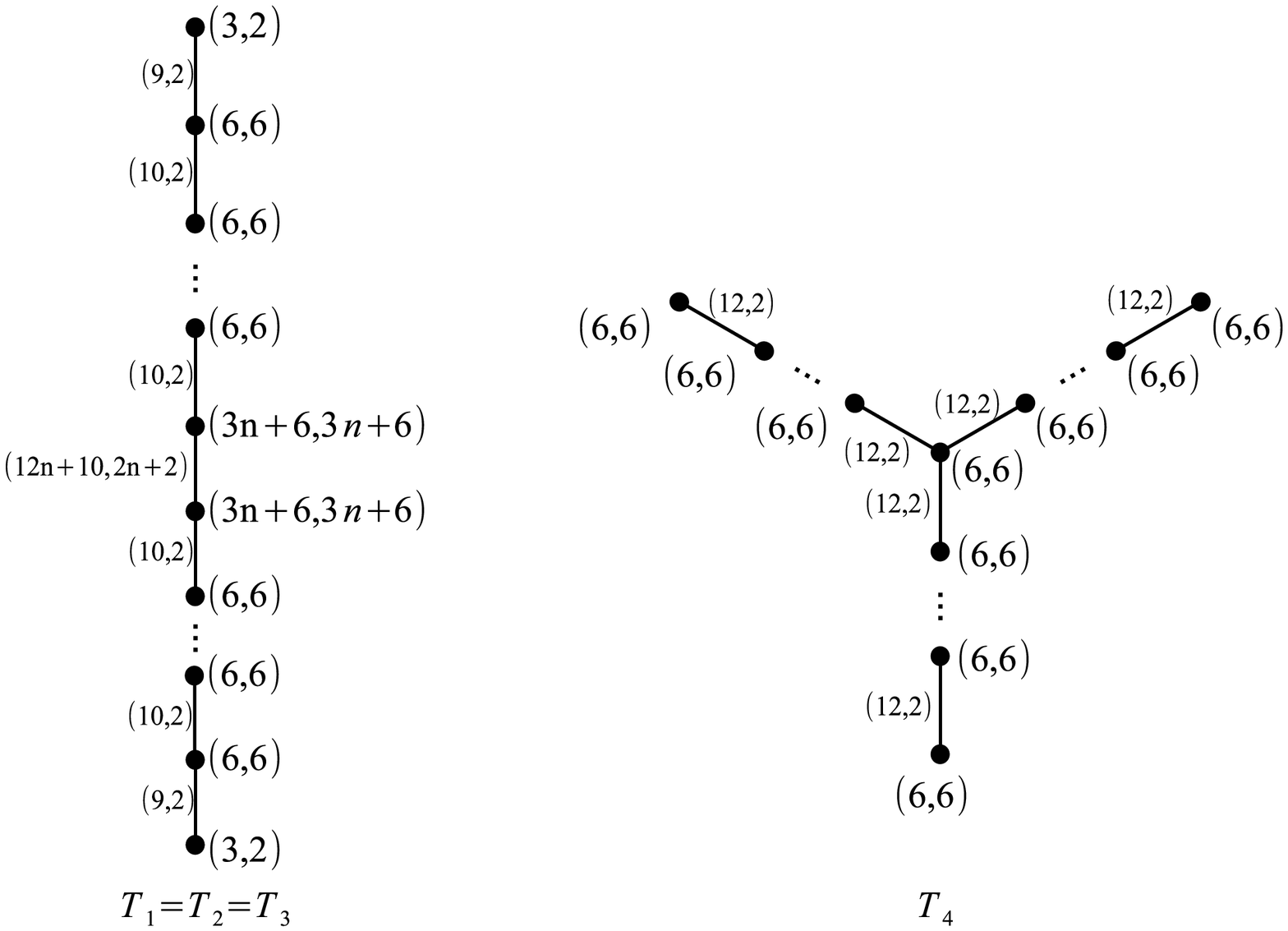}
\end{center}
\caption{\label{phenylenes_quotient} Quotient trees $T_1=T_2=T_3$ and $T_4$ of phenylene $Ph_n$.}
\end{figure}

\begin{eqnarray*}
w^+Sz_e^*(T_1) & = & 2 \cdot 9 \cdot \left(0 + 2 + \frac{2}{2} \right) \left( 6n + (18n +2) + \frac{2}{2} \right) \\
& + & 2 \cdot \sum_{i=2}^n  \Bigg( 10 \left((2i-2) + (6i-4) + \frac{2}{2} \right) \\
& \cdot &  \left( (6n-2i+2) + (18n - 6i + 8) + \frac{2}{2} \right) \Bigg) \\
& + &  (12n+10)\left(2n + (9n+2) + \frac{2n+2}{2} \right) \left( 2n +  (9n+2) + \frac{2n+2}{2} \right) \\
& = & \frac{1}{3} \left( 9664n^3 + 7392n^2 + 1952n + 216 \right),
\end{eqnarray*}
\begin{eqnarray*}
w^+Sz_e^*(T_4) & = & 3 \cdot \sum_{i=1}^n 12 \left((2i-2) + 6i + \frac{2}{2} \right) \left((6n-2i) + (18n-6i+6) + \frac{2}{2} \right) \\
& = & 2688n^3 + 2592n^2 + 516n.
\end{eqnarray*}

\noindent
Therefore, by Proposition \ref{izr_phe} we get
\begin{eqnarray*}
w^+Sz_e^*(Ph_n) & = & 3 \cdot w^+Sz_e^*(T_1) + w^+Sz_e^*(T_4)  \\
& = & 12352n^3 + 9984n^2 + 2468n + 216.
\end{eqnarray*}

\noindent
Next, one can compute the square vertex-PI index of strength-weighted quotient trees:
\begin{eqnarray*}
PI_v^s(T_1) & = &  2 \cdot \sum_{i=1}^n 2 \left(  \left( 6i-3 \right)^2  +   \left( 18n -6i + 9 \right)^2 \right) \\
& + &  (2n+2) \left(  \left( 9n+3 \right)^2  +   \left( 9n+3 \right)^2 \right) \\
& = & 1284n^3 + 1260n^2 + 372n + 36,
\end{eqnarray*}
\begin{eqnarray*}
PI_v^s(T_4) & = &  3 \cdot \sum_{i=1}^n 2 \left(  \left( 6i \right)^2  +   \left( 18n - 6i + 6 \right)^2 \right) \\
& = & 1440n^3 + 648n^2 + 72n.
\end{eqnarray*}

\noindent
Again, by Proposition \ref{izr_phe} we obtain
\begin{eqnarray*}
PI_v^s(Ph_n) & = & 3 \cdot PI_v^s(T_1) + PI_v^s(T_4)  \\
& = & 5292n^3 + 4428n^2 + 1188n + 108.
\end{eqnarray*}

Similarly as in the previous section, by Proposition \ref{izr_phe} we obtain the following computational result.
\begin{proposition}
Let $G_{sw}$ be a strength-weighted phenylene with $n$ vertices. If $F$ is a normal function for $G_{sw}$, then the index $TI_F(G_{sw})$ can be computed in $O(n)$ time.
\end{proposition}

\subsection{Coronoid systems}

Coronoid hydrocarbons are, like benzenoid hydrocarbons, polycyclic molecules composed of hexagonal rings \cite{gucy-89}. Their mathematical models, known as coronoid systems, are often regarded as benzenoid systems that are allowed to have holes. Formally, we take two
cycles, $C$ and $C'$, in the hexagonal lattice where $C'$ is completely embraced by $C$ and the size of $C'$ is greater than 6. A \textit{coronoid system} consists of the vertices and edges on $C$ and $C'$, and also of the vertices and edges that lie outside $C'$ but in the interior of $C$. The vertices and edges on $C'$ and its interior are sometimes referred to as the \textit{corona hole}. For more information see Chapter 8 in \cite{gucy-89}.

In this subsection, we consider a family of coronoid systems with a fixed corona hole. In particular, $Co_n$, $n \geq 1$, denotes the coronoid system with $n$ layers of hexagons around the hole, see Figure \ref{coronoid}. Again, our goal is to deduce closed-form formulas for the weighted-plus revised edge-Szeged index and the square vertex-PI index for this family of molecular graphs.

It is easy to calculate $|V(Co_n)| = 6n^2 + 24 n + 18$ and $|E(Co_n)| = 9n^2 + 33 n + 18$. The three main representatives of $\Theta^*$-classes are also shown in Figure \ref{coronoid}. These $\Theta^*$-classes will be denoted as $E_{1r}$, $r \in \lbrace 1, \ldots, n \rbrace$, $E_2$ and $E_3$. Since $E_2$ is a $\Theta^*$-class but not a $\Theta$-class, we can see that relation $\Theta$ is not transitive and therefore, $Co_n$ is not a partial cube, which makes our example more complex.
\begin{figure}[h]
\begin{center}
\includegraphics[scale=0.7,trim=0cm 0.4cm 0cm 0cm]{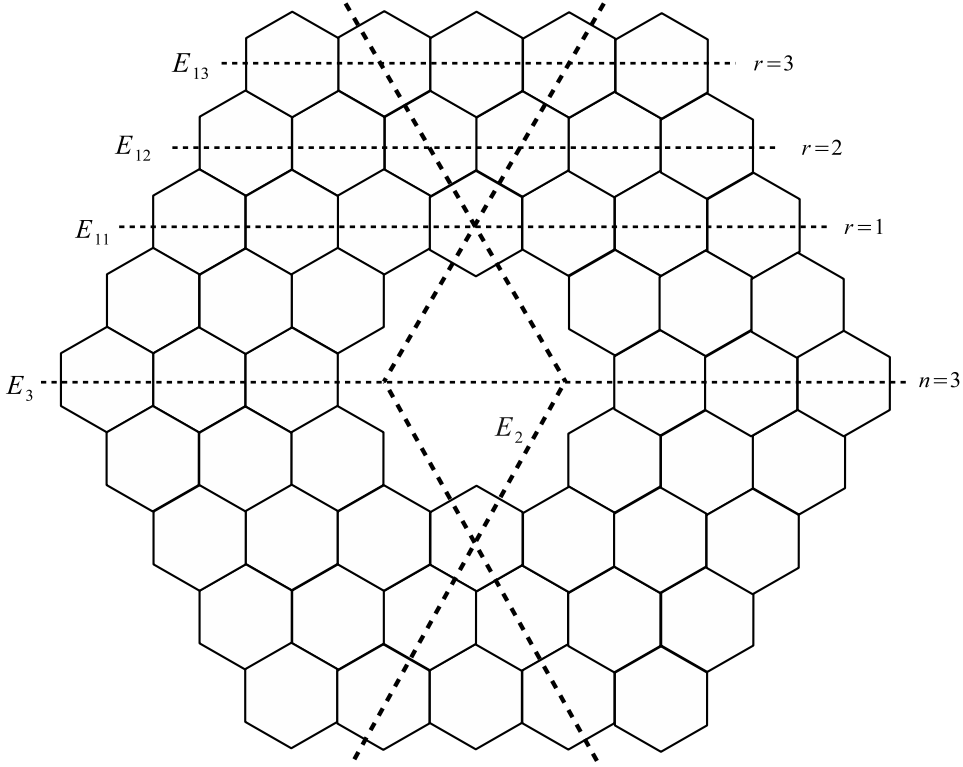}
\end{center}
\caption{\label{coronoid} Coronoid system $Co_3$ together with the representatives of $\Theta^*$-classes.}
\end{figure}

\noindent
The corresponding quotient graphs $G_{1r}=Co_n / E_{1r}$, $G_2 = Co_n/ E_2$, and $G_3 = Co_n / E_3$ are depicted in Figure \ref{coronoid_quotient}. The same notation will be used for the strength-weighted quotient graphs. Again, we suppose that graph $Co_n$ is normally strength weighted with $w_e=w_e^+$. Therefore, as in the previous subsections, we use the weight $s_e$ instead of $w_e$ when considering the square vertex-PI index.

\begin{figure}[h!]
\begin{center}
\includegraphics[scale=0.8,trim=0cm 0.4cm 0cm 0cm]{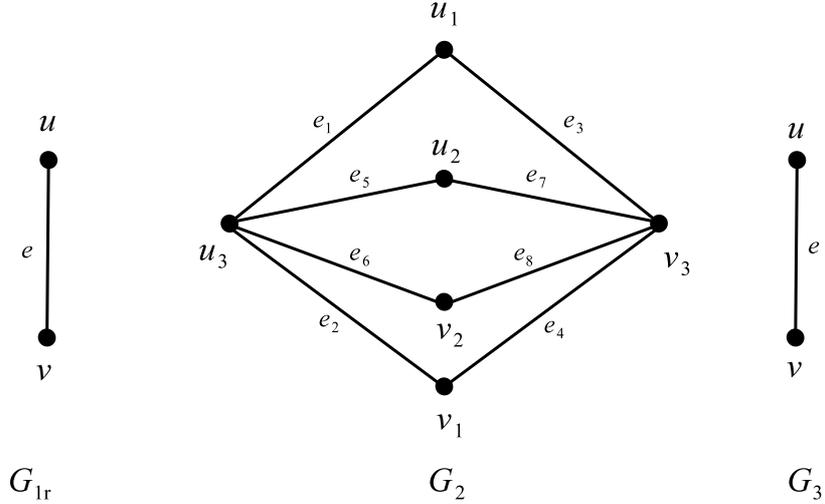}
\end{center}
\caption{\label{coronoid_quotient} Quotient graphs $G_{1r}$, $G_2$, and $G_3$ for coronoid system $Co_n$.}
\end{figure}

Firstly, we calculate the indices for the quotient graph $G_{1r}$, $r \in \lbrace 1, \ldots, n \rbrace$. The weights of $G_{1r}$ are:
\begin{eqnarray*}
w_v^1 (u) & = & 3n^2 - 4nr + 8n + r^2 - 6r + 5,  \\
w_v^1(v) &= & 3n^2 + 4nr - r^2 + 16n + 6r + 13, \\
s_v^1 (u) & = & \frac{1}{2} \left(9n^2 - 12nr + 3r^2 + 19n - 15r + 8 \right), \\
s_v^1 (v) & = & \frac{1}{2} \left( 9n^2 + 12nr - 3r^2 + 43n + 17r + 22 \right), \\
w_e^1(e) & = & 12n - 6r + 16, \\
s_e^1(e) & = & 2n - r + 3.
\end{eqnarray*}

\noindent
Therefore, the corresponding indices of $G_{1r}$ can be computed:
\begin{eqnarray*}
w^+Sz_e^* (G_{1r}) & = & w_e^1(e) \left( s_v^1 (u) + \frac{s_e^1(e)}{2} \right) \left( s_v^1 (v) + \frac{s_e^1(e)}{2} \right) \\
& = & \frac{1}{2} \big( 486n^5 - 243n^4r - 864n^3r^2 + 864n^2r^3 - 270nr^4 \\ &+& 27r^5 + 4212n^4 - 3510n^3r - 2160n^2r^2 + 2088nr^3 - 360r^4 \\ &+& 12366n^3 - 9423n^2r - 2124nr^2 + 1410r^3 + 16272n^2 - 8820nr  \\ &+& 9810n - 1040r^2 - 2617r + 2200 \big)
\end{eqnarray*}
\noindent and
\begin{eqnarray*}
PI_v^s (G_{1r}) & = & s_e^1(e) \left( \left( w_v^1(u) \right)^2 + \left( w_v^1(v) \right)^2 \right) \\
& = &  \big( 36n^5 - 18n^4r + 64n^3r^2 - 64n^2r^3 + 20nr^4 - 2r^5 + 342n^4 - 16n^3r \\
  &+ & 192n^2r^2 - 176nr^3 + 30r^4 + 1288n^3 + 84n^2r + 192nr^2 - 128r^3 \\
  & + & 2276n^2 + 176nr + 72r^2 + 1876n + 94r + 582 \big).
\end{eqnarray*}


Next, we calculate the weights for strength-weighted quotient graph $G_2$:
\begin{eqnarray*}
w_v^2 (u_1)& = & w_v^2(v_1)  = n^2,  \\
s_v^2 (u_1) & = & s_v^2 (v_1) = \frac{1}{2} \left(3n^2-3n \right), \\
w_v^2 (u_2)& = & w_v^2(v_2)  =  1,  \\
s_v^2 (u_2) & = & s_v^2 (v_2) = 0, \\
w_v^2 (u_3)& = & w_v^2(v_3)  =  2n^2 + 12n + 8,  \\
s_v^2 (u_3) & = & s_v^2 (v_3) = 3n^2 + 16n + 7, \\
w_e^2(e_1) & = & w_e^2(e_2) = w_e^2(e_3) = w_e^2(e_4)= 6n-1, \\
s_e^2(e_1) & = & s_e^2(e_2) = s_e^2(e_3) = s_e^2(e_4)= n,\\
w_e^2(e_5) & = & w_e^2(e_6) = w_e^2(e_7) = w_e^2(e_8)= 5, \\
s_e^2(e_5) & = & s_e^2(e_6) = s_e^2(e_7) = s_e^2(e_8)= 1.
\end{eqnarray*}

\noindent
Hence, the corresponding indices of $G_2$ are:
\begin{eqnarray*}
w^+Sz_e^* (G_2) & = & 4 (6n-1) \left(3n^2 + 16n + 7 + \frac{1}{2}(3n^2-3n)  + n + 2 +\frac{2n+2}{2} \right) \\
& \cdot &  \left( \frac{1}{2}(3n^2-3n) + 3n^2 + 16n + 7 + n +  \frac{2n+2}{2} \right) \\
& + & 4 \cdot 5 \cdot \left(3n^2 + 16n + 7 + 3n^2-3n  + 2n + 1 +\frac{2n+2}{2} \right) \\
& \cdot &  \left( 3n^2 + 16n + 7 + 1 +  \frac{2n+2}{2} \right) \\
& = & 486n^5 + 3843n^4 + 10884n^3 + 12775n^2 + 6672n + 1300
\end{eqnarray*}
\noindent and
\begin{eqnarray*}
PI_v^s (G_2) & = & 4 n \left ( \left( 2n^2 + 12n + 8 + n^2 + 2 \right)^2 + \left( n^2 + 2n^2 + 12n + 8 \right)^2 \right) \\
&+& 4  \left ( \left( 2n^2 + 12n + 8 + 2n^2 + 1 \right)^2 + \left( 1 + 2n^2 + 12n + 8 \right)^2 \right) \\
& = &  72n^5 + 656n^4 + 2160n^3 + 3312n^2 + 2384n + 648.
\end{eqnarray*}

Furthermore, we consider quotient graph $G_3$. The weights for this strength-weighted graph are:
\begin{eqnarray*}
w_v^3 (u)& = & w_v^3(v)  =  3n^2 + 12n+9,  \\
s_v^3 (u) & = & s_v^3 (v) = \frac{1}{2} \left( 9n^2 + 31n +16 \right), \\
w_e^3(e) & = & 12n+8, \\
s_e^3(e) & = & 2n+2.
\end{eqnarray*}

\noindent
Therefore, we obtain the corresponding indices for $G_3$:
\begin{eqnarray*}
w^+Sz_e^* (G_3) & = & w_e^3(e) \left( s_v^3 (u) + \frac{s_e^3(e)}{2} \right) \left( s_v^3 (v) + \frac{s_e^3(e)}{2} \right) \\
& = & 243n^5 + 1944n^4 + 5427n^3 + 6390n^2 + 3348n + 648,
\end{eqnarray*}

\begin{eqnarray*}
PI_v^s (G_3) & = & s_e^3(e) \left( \left( w_v^3(u) \right)^2 + \left( w_v^3(v) \right)^2 \right)  \\
& = &  36n^5 + 324n^4 + 1080n^3 + 1656n^2 + 1188n + 324.
\end{eqnarray*}


Finally, by using the main result, Theorem \ref{glavni}, we deduce
\begin{eqnarray*}
w^+Sz_e^* (Co_n) & = & 6  \cdot \sum_{r=1}^n w^+Sz_e^*(G_{1r}) + 3 \cdot w^+Sz_e^*(G_2) + 3 \cdot w^+Sz_e^*(G_3) \\
& = & \frac{1}{4} \big( 2916n^6 + 32076n^5 + 133299n^4 + 273598n^3 \\
&+&  272901n^2 + 129002n + 23376 \big)
\end{eqnarray*}
\noindent and
\begin{eqnarray*}
PI_v^s (Co_n) & = & 6 \cdot \sum_{r=1}^n PI_v^s(G_{1r}) + 3 \cdot PI_v^s(G_2) + 3 \cdot PI_v^s(G_3) \\
& = & 216n^6 + 2484n^5 + 11205n^4 + 24480n^3 \\
&+& 27183n^2 + 14556n + 2916.
\end{eqnarray*}

\section{Conclusion}

In the paper we developed a cut method for computing  Szeged-like topological indices, which are some of the most investigated distance-based molecular descriptors. This method reduces the problem of calculating a topological index of a strength-weighted graph to the problem of computing the topological index of corresponding strength-weighted quotient graphs. As an example, we determined two Szeged-like topological indices for some benzenoid systems, phenylenes and coronoid systems, which are important and well-known classes of molecular graphs. In the case of benzenoid systems and phenylenes our method enables us to find the value of the topological index by using strength-weighted quotient trees, which leads to efficient algorithms. Regarding the future work, our main result can be applied to calculate Szeged-like topological indices of various molecular nanostructures and to deduce closed-form formulas for infinite families of such graphs.

\section*{Funding information} 

\noindent The author Niko Tratnik acknowledge the financial support from the Slovenian Research Agency (research core funding No. P1-0297 and J1-9109).


\begin{thebibliography}{99}

\bibitem{kl-2020} Y. Alizadeh, S. Klav\v zar, Complexity of the Szeged index, edge orbits, and some nanotubical fullerenes, Hacet. J. Math. Stat. 49 (2020) 87--95. 

\bibitem{aroc} M. Arockiaraj, J. Clement, K. Balasubramanian, Topological indices and their applications to circumcised donut benzenoid systems, Kekulenes and drugs, Polycycl.
Aromat. Comp. 40 (2020) 280--303.

\bibitem{aroc1} M. Arockiaraj, J. Clement, D. Paul, K. Balasubramanian, Quantitative structural descriptors of sodalite materials, J. Mol. Struct. 1223 (2021) 128766.

\bibitem{aroc_cle_trat} M. Arockiaraj, J. Clement, N. Tratnik, Mostar indices of carbon nanostructures and circumscribed donut benzenoid systems, Int. J. Quantum Chem. 119 (2019) e26043.

\bibitem{aroc2} M. Arockiaraj, J. Clement, N. Tratnik, S. Mushtaq,
K. Balasubramanian, Weighted Mostar indices as measures of molecular peripheral shapes with applications to graphene, graphyne and graphdiyne nanoribbons, SAR QSAR Environ. Res. 31 (2020) 187--208.

\bibitem{aroc3} M. Arockiaraj, S. Klav\v zar, S. R. J. Kavitha, S. Mushtaq, K. Balasubramanian, Relativistic structural characterization of molybdenum and tungsten disulfide materials, Int. J. Quantum Chem., e26492 (https://doi.org/10.1002/qua.26492).

\bibitem{aroc4} M. Arockiaraj, S. Klav\v zar, S. Mushtaq, K. Balasubramanian, Topological indices of the subdivision of a family of partial cubes and computation of $\rm{SiO}_2$ related structures, J. Math. Chem. 57 (2019) 1868--1883.


  
   \bibitem{bok} J. Bok, B. Furtula, N. Jedli\v ckov\' a, R. \v Skrekovski, On extremal graphs of weighted Szeged index, {MATCH Commun. Math. Comput. Chem.} {82} (2019) 93--109.
   
     \bibitem{chazelle} B.~Chazelle,  Triangulating a simple polygon in linear time,
	{Discrete Comput. Geom.} {6} (1991) 485--524.
   
\bibitem{chepoi-1996}
  V.~Chepoi, 
  On distances in benzenoid systems, {J. Chem. Inf. Comput. Sci.} {36} (1996) 1169--1172.

\bibitem{chepoi-1997}
  V.~Chepoi, S.~Klav\v zar,
  The Wiener index and the Szeged index of benzenoid systems in linear time, {J. Chem. Inf. Comput. Sci.} {37} (1997) 752--755.
  
  \bibitem{CK-1998}
  V.~Chepoi, S.~Klav\v zar,
	Distances in benzenoid systems: Further developments, 
	Discrete Math.	192 (1998) 27--39.
  
    
    \bibitem{deng} K. Deng, S. Li, On the extremal values for the Mostar index of trees with given degree sequence, Appl. Math. Comput. 390 (2021) 125598.
    
  \bibitem{mostar} T. Do\v sli\' c, I. Martinjak, R. \v Skrekovski, S. Tipuri\' c Spu\v zevi\' c, I. Zubac, Mostar index, J. Math. Chem. 56 (2018) 2995--3013. 
  
  \bibitem{estrada} E. Estrada, The Structure of Complex Networks, Oxford University Press, New York, 2011.
  
\bibitem{ghorbani} M. Ghorbani, X. Li, H. R. Maimani, Y. Mao, S. Rahmani, M. Rajabi-Parsa, Steiner (revised) Szeged index of graphs, MATCH Commun. Math. Comput. Chem. 82 (2019) 733--742.
  
  
  
%


\bibitem{gut_sz} I. Gutman, A formula for the Wiener number of trees and its extension to graphs containing cycles, {Graph Theory Notes N. Y.} {27} (1994) 9--15. 

\bibitem{gut}
I. Gutman, A. R. Ashrafi, The edge version of the Szeged Index, {Croat. Chem. Acta.} {81}(2) (2008) 263--266.

\bibitem{gucy-89}
I.  Gutman, S.~J. Cyvin, 
{Introduction to the Theory of Benzenoid Hydrocarbons\/}, Springer-Verlag, Berlin, 1989.




\bibitem{klavzar-book}
R. Hammack, W. Imrich, S. Klav\v{z}ar, \textit{Handbook of Product Graphs}, CRC Press, Boca Raton, 2011.

\bibitem{he} S. He, R.-X. Hao, Y.-Q. Feng, On the edge-Szeged index of unicyclic graphs with perfect matchings, Discrete Appl. Math. 284 (2020) 207--223.

\bibitem{huang} S. Huang, S. Li, M. Zhang, On the extremal Mostar indices of hexagonal chains, MATCH Commun. Math. Comput. Chem. 84 (2020) 249--271.

\bibitem{imran} M. Imran, S. Akhter, Z. Iqbal, Edge Mostar index of chemical structures and nanostructures using graph operations, Int. J. Quantum Chem. 120 (2020) e26259.

\bibitem{ilic} A. Ili\'{c}, N. Milosavljevi\'{c}, The weighted vertex PI index, {Math. Comput. Model.} {57} (2013) 623--631.









%
%
%
%




  
%






\bibitem{def_pi} P. V. Khadikar, On a novel structural descriptor PI, {Nat. Acad. Sci. Lett.} {23} (2000) 113--118.

\bibitem{drug} P. V. Khadikar, S. Karmarkar, V. K. Agrawal, J. Singh, A. Shrivastava, I. Lukovits, M. V. Diudea, Szeged index - applications for drug modeling, {Lett. Drug. Des. Discov.} {2} (2005) 606--624.

\bibitem{khal1} M. H. Khalifeh, H. Yousefi-Azari, A. R. Ashrafi, Vertex and edge PI indices of Cartesian product graphs, {Discrete Appl. Math.} {156} (2008) 1780--1789.

\bibitem{kl-2018} S. Klav\v zar, S. Li, H. Zhang, On the difference between the (revised) Szeged index and the Wiener index of cacti, Discrete Appl. Math. 247 (2018) 77--89.

\bibitem{nad_klav} S. Klav\v zar, M. J. Nadjafi-Arani, Wiener index in weighted graphs via unification
of $\Theta^*$-classes, { European J. Combin.} {36} (2014) 71--76.

\bibitem{klavzar-2015}
S. Klav\v zar, M. J. Nadjafi-Arani,  
Cut method: update on recent developments and equivalence of independent approaches,
{Curr. Org. Chem.}  {19} (2015) 348--358. 

\bibitem{li} X. Li, M. Zhang, A Note on the computation of revised (edge-)Szeged index in terms of canonical isometric embedding, MATCH Commun. Math. Comput. Chem. 81 (2019) 149--162.

\bibitem{li1} X. Li, M. Zhang, Results on two kinds of Steiner distance-based indices for some classes of graphs, MATCH Commun. Math. Comput. Chem. 84 (2020) 567--578.

\bibitem{liu} M. Liu, K. C. Das, On the Steiner (revised) Szeged index, MATCH Commun. Math. Comput. Chem. 84 (2020) 579--594.

\bibitem{ma2} G. Ma, Q. Bian, J. Wang, The weighted vertex PI index of $(n,m)$-graphs with given diameter, {Appl. Math. Comput.} {354} (2019) 329--337.

\bibitem{pisanski_rand} T. Pisanski, M. Randi\' c, Use of the Szeged index and the revised Szeged index for measuring network bipartivity, Discrete Appl. Math. 158 (2010) 1936--1944.

\bibitem{randic} M. Randi\' c, On generalization of Wiener index for cyclic structures, Acta Chim. Slovenica 49 (2002) 483--496.



\bibitem{tepeh} A. Tepeh, Extremal bicyclic graphs with respect to Mostar index, Appl. Math. Comput. 355 (2019) 319--324.








%
%

\bibitem{tratnik_mostar} N. Tratnik, Computing the Mostar index in networks with applications to molecular graphs, arXiv:1904.04131.

\bibitem{tratnik_weighted_sz} N. Tratnik, Computing weighted Szeged and PI indices from quotient graphs, Int. J. Quantum Chem. 119 (2019) e26006.

\bibitem{tratnik_edge_sz} N. Tratnik, The edge-Szeged index and the PI index of benzenoid systems in linear time, {MATCH Commun. Math. Comput. Chem.} {77} (2017) 393--406.


\bibitem{wiener} H. Wiener, Structural determination of paraffin boiling points, {J. Amer. Chem. Soc.} {69} (1947) 17--20.

\bibitem{zigert-2018} P. \v Zigert Pleter\v sek, The edge-Wiener index and the edge-hyper-Wiener index of phenylenes, {Discrete Appl. Math.} {255} (2019) 326--333.


  
    
\end{thebibliography}
\end{document}